\documentclass[a4paper,reqno]{amsart}
\usepackage{amssymb, amsmath, amscd}
\usepackage[final]{graphicx}
\usepackage{float}\usepackage{wrapfig}
\usepackage{bbm}
\usepackage[final]{graphicx}

\def\ad{\operatorname{ad}}

\newtheorem{theorem}{Theorem}[section]
\newtheorem{ansatz}[theorem]{Ansatz}
\newtheorem{lemma}[theorem]{Lemma}

\makeatletter
 \@addtoreset{equation}{section}
\makeatother
\begin{document}
\title[A short proof]
{On Distinguished
	local coordinates for locally homogeneous affine surfaces}
\author{M. Brozos-V\'{a}zquez \, E. Garc\'{i}a-R\'{i}o\, P. Gilkey}
\address{MBV: Universidade da Coru\~na, Differential Geometry and its Applications Research Group, Escola Polit\'ecnica Superior, 15403 Ferrol,  Spain}
\email{miguel.brozos.vazquez@udc.gal}
\address{EGR: Faculty of Mathematics, University of Santiago de Compostela, 15782 Santiago de Compostela, Spain}
\email{eduardo.garcia.rio@usc.es}
\address{PBG: Mathematics Department, University of Oregon, Eugene OR 97403-1222, USA}
\email{gilkey@uoregon.edu}
\thanks{Supported by project MTM2016-75897-P (Spain).}
\begin{abstract}{We give a new short self-contained proof of the result of Opozda~\cite{Op04}
classifying the locally homogeneous torsion free affine surfaces and the extension to the case of surfaces
with torsion due to Arias-Marco and Kowalski~\cite{AMK08}. Our approach rests on a direct analysis of the
affine Killing equations and is quite different than the approaches taken previously
in the literature.}
\end{abstract}
\maketitle

\section{Introduction}

We say that $\mathcal{M}=(M,\nabla)$ is an {\it affine surface}
if $M$ is a smooth connected 2-dimensional manifold and if $\nabla$ is a connection on the tangent bundle of $M$. 
We emphasize that $\nabla$ is permitted to have torsion.
We say that $\mathcal{M}$ is {\it locally homogenous} if given any
two points of $M$, there is a local diffeomorphism from a neighborhood of one point to a neighborhood of the other
point which preserves $\nabla$, i.e. is an affine map.  In a system of local coordinates, sum over repeated indices to expand
 $\nabla_{\partial_{x^i}}\partial_{x^j}=\Gamma_{ij}{}^k \partial_{x^k}$ to define the Christoffel symbols.

During the past few years, there has been a concerted effort to classify homogeneous affine surfaces. 
Kowalski, Opozda and Vl\'a\v sek \cite{KVOp2} provided the first major step by classifying the homogeneous
torsion free connections with skew-symmetric Ricci tensor.
Derdzinski \cite{D08} then extended their result using in an essential
fashion the fact that the curvature operator satisfies  the identity $R(x,y)=\rho(x,y)\operatorname{Id}$
in this setting.
Subsequently, Opozda~\cite{Op04} established a complete classification for locally homogeneous surfaces without
torsion. Finally Arias-Marco and Kowalski~\cite{AMK08} completed the program by extending
the Theorem of Opozda  to connections with torsion.
 The resulting full classification can be stated as follows.

\goodbreak\begin{theorem}\label{T1.1}
If $\mathcal{M}$ is a locally homogeneous affine surface, then at least one of the following three possibilities holds describing the local geometry:
\begin{enumerate}
\item There exists a coordinate atlas so that $\Gamma_{ij}{}^k\in\mathbb{R}$.
\item There exists a coordinate atlas so that
$\Gamma_{ij}{}^k=(x^1)^{-1}A_{ij}{}^k$ for $A_{ij}{}^k\in\mathbb{R}$.
\item There exists a coordinate atlas where
$\nabla$ is the Levi-Civita connection defined by the metric of the round sphere.\end{enumerate}
\end{theorem}

The compact case was considered in \cite{AG14, Opozda}. If $M$ is compact, then either $\nabla$ is torsion-free and flat, $\nabla$ is the Levi-Civita connection of a surface of constant curvature, or $\nabla$ is a connection with
$\Gamma_{ij}{}^k\in\mathbb{R}$ and $M$ is a torus.
The special case of locally symmetric affine surfaces was addressed in \cite{Opozda1991}, where it is shown that 
any locally symmetric affine surface is either modeled on a surface of constant curvature with the Levi-Civita connection or, up to linear equivalence, on one of two affine surfaces which have
the form given in Theorem~\ref{T1.1}-(1).
Theorem~\ref{T1.1} has been useful in many works on affine surfaces, 
including but not limited to \cite{D13,D09,KVOp3,KS14}. We also refer to Kowalski et al.~\cite{KVOp4} for another proof of
Theorem~\ref{T1.1} in the torsion free setting.

We shall give a short and self-contained proof of Theorem~\ref{T1.1} by
examining the affine Killing equations rather than by studying the curvature tensor
or by using classification results of Lie algebras of vector fields. The structure of the Lie algebra of affine
Killing vector fields $\mathfrak{K}(\mathcal{M})$ will play a crucial role in our analysis. We choose
coordinate systems so that the vector field $\partial_{x^2}$ is an affine Killing vector field. We complexify and
consider the generalized eigenspaces of $\mathfrak{K}_{\mathbb{C}}(\mathcal{M})$ as an
$\operatorname{ad}(\partial_{x^2})$ module. What is new in this approach is the mixture of Lie theory
together with the affine Killing equations that affords, we believe, a more direct and conceptual approach to the
proof of Theorem~\ref{T1.1}.

\section{Affine Killing vector fields}
We recall the following result of
Kobayashi and Nomizu~\cite[Chapter VI]{KN63}.

\goodbreak
\begin{lemma}\label{L2.1} Let $\mathcal{M}=(M,\nabla)$.
\begin{enumerate}
\item The following 3 conditions are equivalent and if any is satisfied, then $X$ is said to be
an {\rm affine Killing vector field}.
\begin{enumerate}
\item Let $\Phi_t^X$ be the local flow of $X$. Then $(\Phi_t^X)_*\circ\nabla=\nabla\circ(\Phi_t^X)_*$.
\item The Lie derivative of $\nabla$ with respect to $X$ vanishes.
\item $[X,\nabla_YZ]-\nabla_Y[X,Z]-\nabla_{[X,Y]}Z=0$ for all $Y,Z\in C^\infty(TM)$.
\end{enumerate}
\item Let $\mathfrak{K}(\mathcal{M})$ be the vector space of affine Killing vector fields.
The Lie bracket gives $\mathfrak{K}(\mathcal{M})$ a Lie algebra structure. We have that
$\dim\{\mathfrak{K}(\mathcal{M})\}\le 6$. 
\end{enumerate}
\end{lemma}

Let $X=a^k\partial_{x^k}$. By Lemma~\ref{L2.1}~(1c), $X$ is an affine Killing vector field if and only if
$X$ satisfies the 8 affine Killing equations for $1\le i,j,k\le 2$
$$
K_{ij}{}^k:\quad 0=\frac{\partial^2a^k}{\partial_{x^i}\partial_{x^j}}+\sum_\ell\left\{
a^\ell\frac{\partial\Gamma_{ij}{}^k}{\partial x^\ell}-\Gamma_{ij}{}^l\frac{\partial{a^k}}{\partial x^\ell}
+\Gamma_{i\ell}{}^k\frac{\partial a^\ell}{\partial x^j}+\Gamma_{\ell j}{}^k\frac{\partial a^\ell}{\partial_{x^i}}
\right\}\,.
$$

Choose a point $P$ of $M$; which point is irrelevant as we shall assume
that $\mathcal{M}$ is locally homogeneous henceforth. 
We work at the level of germs and
assume $M$ is an arbitrarily small neighborhood of $P$. If, for example, we are given a
vector field $\Xi$ which does not vanish identically near $P$,
we can choose a slightly different base point $\tilde P$ where $\Xi(\tilde P)\ne0$. 
To pass to global results, we shall assume the underlying manifold $M$ is simply connected to avoid
difficulties with holonomy; in this setting, every Killing vector field which is locally defined extends to a
globally defined Killing vector field.
We shall not belabor these points in what follows.
We say that a subset $S$ of $\mathfrak{K}(\mathcal{M})$ is {\it effective} if
there exist $X_i\in S$ so that $\{X_1(P),X_2(P)\}$ are linearly independent. Since $\mathcal{M}$
is locally homogeneous, $\mathfrak{K}(\mathcal{M})$ is effective \cite{Hall, Nomizu}. We define the following Lie algebras by their relations
\begin{equation}\label{E2.a}\begin{array}{l}
\mathfrak{K}_{\mathcal{A}}:=\operatorname{Span}\{X,Y\}\text{ for }[X,Y]=0,\ 
\mathfrak{K}_{\mathcal{B}}:=\operatorname{Span}\{X,Y\}\text{ for }[X,Y]=Y,\\
\mathfrak{so}(3):=\operatorname{Span}\{X,Y,Z\}\text{ for }[X,Y]=Z,\ [Y,Z]=X,\ 
[Z,X]=Y\,.
\end{array}\end{equation}

Theorem~\ref{T1.1} will be a consequence of the following result.

\begin{lemma}\label{L2.2}
Let $\mathcal{M}=(M,\nabla)$ be locally homogeneous and simply connected. 
\begin{enumerate}
\item There is an effective Lie subalgebra $\tilde{\mathfrak{K}}$ of $\mathfrak{K}(\mathcal{M})$
which is isomorphic to $\mathfrak{K}_{\mathcal{A}}$, $\mathfrak{K}_{\mathcal{B}}$, or $\mathfrak{so}(3)$.
\item If $\tilde{\mathfrak{K}}\approx\mathfrak{K}_{\mathcal{A}}$, then
there is a coordinate atlas so that $\Gamma_{ij}{}^k\in\mathbb{R}$.
\item If $\tilde{\mathfrak{K}}\approx\mathfrak{K}_{\mathcal{B}}$, then
there is a coordinate atlas so that
$\Gamma_{ij}{}^k=(x^1)^{-1}A_{ij}{}^k$ for $A_{ij}{}^k\in\mathbb{R}$.
\item If $\tilde{\mathfrak{K}}\approx\mathfrak{so}(3)$, then
there is a coordinate atlas where
$\nabla$ is the Levi-Civita connection defined by the metric of the round sphere.
\end{enumerate}\end{lemma}

The possibilities of Assertion~(2) and Assertion~(3) are not exclusive; 
the non-flat examples such that both Assertion~(2) and Assertion~(3) hold along with a complete
description of the Lie algebras $\mathfrak{K}(\mathcal{M})$
are given in \cite{BGG18}.
By contrast, any $\mathcal{M}$ which admits an effective 
$\mathfrak{so}(3)$ Lie subalgebra of $\mathfrak{K}(\mathcal{M})$ satisfies
$\dim\{\mathfrak{K}(\mathcal{M})\}=3$ and
does not admit any 2-dimensional Lie subalgebras of affine Killing vector fields.
In Theorem~\ref{T1.1}, we do not impose the hypothesis that $\mathcal{M}$ is simply
connected as the question of suitable coordinate systems is a local one. By contrast, in Lemma~\ref{L2.2},
we must impose the hypothesis that $\mathcal{M}$ is simply connected since the question of affine Killing
vector fields is a global one.

By Lemma~\ref{L2.1}, 
$\dim\{\mathfrak{K}(\mathcal{M})\}\le6$. 
Complexify and set $\mathfrak{K}_{\mathbb{C}}(\mathcal{M}):=\mathfrak{K}(\mathcal{M})\otimes_{\mathbb{R}}\mathbb{C}$.

\begin{lemma}\label{L2.3} Choose $\Xi\in\mathfrak{K}(\mathcal{M})$ with $\Xi(P)\ne0$.
For $\alpha\in\mathbb{C}$, let 
\smallbreak\centerline{$E(\alpha):=\{X_\alpha\in\mathfrak{K}_{\mathbb{C}}(\mathcal{M}):(\ad(\Xi)-\alpha)^6X_\alpha=0\}$}
\smallbreak\noindent
be the associated generalized eigenspace of $\ad(\Xi)$. Then $[E(\alpha),E(\beta)]\subset E(\alpha+\beta)$.
\end{lemma}

\begin{proof} Choose local coordinates so $\Xi=\partial_{x^2}$.
Then $X_\alpha\in E(\alpha)$ if and only if
\begin{equation}\label{E2.b}
X_\alpha=e^{\alpha x^2}\left\{\sum_{i=0}^{i_0}u_i(x^1)(x^2)^i\partial_{x^1}
+\sum_{j=0}^{j_0}v_{j}(x^1)(x^2)^j\partial_{x^2}\right\}\,.
\end{equation}
This leads to an expansion for $[X_\alpha,X_\beta]$ where the relevant exponential is $e^{(\alpha+\beta)x^2}$ that shows
$[X_\alpha,X_\beta]\in E(\alpha+\beta)$.
\end{proof}

\section{The proof of Lemma~\ref{L2.2}}

The following coordinate normalization will be used for much of our analysis; a different normalization
will be used in examining the structure $\mathfrak{so}(3)$.

\begin{lemma}\label{L3.1}
Let $\Xi\in\mathfrak{K}(\mathcal{M})$ satisfy $\Xi(P)\ne0$. We can choose local coordinates centered
at $P$ so that $\Xi=\partial_{x^2}$ and so that
$$
\Gamma_{ij}{}^k(x^1,x^2)=\Gamma_{ij}{}^k(x^1),\quad
\Gamma_{11}{}^1(x^1)=0,\quad\Gamma_{11}{}^2(x^1)=0\,.
$$
\end{lemma}

\begin{proof}
Choose initial coordinates $(y^1,y^2)$ centered at $P$ so that $\Xi=\partial_{y^2}$. Since $\Xi$
is an affine Killing vector field,
$\Gamma=\Gamma(y^1)$ and the map $(y^1,y^2)\rightarrow(y^1,y^2+t)$ is an affine map. Let
$\sigma(s)$ be a geodesic with $\sigma(0)=0$ and with $\{\dot\sigma(0),\Xi(0)\}$ linearly independent.
Let $T(x^1,x^2):=\sigma(x^1)+(0,x^2)$ define new coordinates with $\partial_{x^2}=\partial_{y^2}$.
Since the curves $x^1\rightarrow T(x^1,x^2)$ are geodesics for $x^2$ fixed and 
since $\partial_{x^2}$ is an affine Killing vector field,
the normalizations of the Lemma hold.\end{proof}

With the coordinate normalization of Lemma~\ref{L3.1}, $\partial_{x^2}$ is a Killing vector field. 
We now examine other Killing vector fields.

\begin{lemma}\label{L3.2} Use Lemma~\ref{L3.1} to normalize the system of local coordinates. 
Set 	
$\mathfrak{K}_\alpha(\mathcal{M}):=\{X=e^{\alpha x^2}v(x^1)\partial_{x^2}:X\in\mathfrak{K}_{\mathbb{C}}(\mathcal{M})\}$.
\begin{enumerate}
\item If there exists  $X\in\mathfrak{K}_\alpha(\mathcal{M})$, which is not a constant multiple of $\partial_{x^2}$, then 
\begin{equation}\label{E3.a}
\Gamma_{11}{}^1=0,\ \Gamma_{11}{}^2=0,\ \Gamma_{12}{}^1=0,\ \Gamma_{21}{}^1=0,\ \Gamma_{22}{}^1=0,\ \Gamma_{22}{}^2=-\alpha.
 \end{equation}
 \item Suppose that the Christoffel symbols satisfy Equation~(\ref{E3.a}).
\begin{enumerate}
\item If $u(x^1,x^2)\partial_{x^1}+w(x^1,x^2)\partial_{x^2}\in\mathfrak{K}_{\mathbb{C}}(\mathcal{M})$, then
\begin{enumerate}
\item$\alpha\   u^{(0,1)}(x^1,x^2)+u^{(0,2)}(x^1,x^2)=0$,
\item$(\Gamma_{12}{}^2(x^1)+\Gamma_{21}{}^2(x^1))w^{(1,0)}(x^1,x^2)+w^{(2,0)}=0$.
\end{enumerate}
\item 
$\mathfrak{K}_\alpha(\mathcal{M})=\{e^{\alpha x^2}v(x^1)\partial_{x^2}:(\Gamma_{12}{}^2(x^1)+\Gamma_{21}{}^2(x^1))v^\prime(x^1)+v^{\prime\prime}(x^1)=0\}$.
\item Assume $\alpha=0$.
If $u(x^1,x^2)\partial_{x^1}+\{\sum_nw_n(x^1)(x^2)^n\}\partial_{x^2}\in\mathfrak{K}_{\mathbb{C}}(\mathcal{M})$, then
 $w_n(x^1)\partial_{x^2}\in\mathfrak{K}_0(\mathcal{M})$ for all $n$. Furthermore,
 $x^2\partial_{x^2}\in\mathfrak{K}(\mathcal{M})$.
\end{enumerate}\end{enumerate}\end{lemma}

\begin{proof} Our choice of coordinate system yields $\Gamma_{11}{}^1=\Gamma_{11}{}^2=0$.  It is convenient to decompose the proof of Assertion~(1) into two cases.
\smallbreak\noindent{\bf Case 1.1.} Suppose $\alpha\ne0$. Assume $0\ne X=e^{\alpha x^2}v(x^1)\partial_{x^2}\in\mathfrak{K}_{\mathbb{C}}(\mathcal{M})$. Equation~(\ref{E3.a}) follows from the equations
\begin{eqnarray*}
&&K_{22}{}^1:\ 0=e^{\alpha x^2}\{2\alpha \Gamma_{22}{}^1(x^1)v(x^1)\}, \text{ so }\Gamma_{22}{}^1(x^1)=0.\\
&&K_{12}{}^1:\ 0=e^{\alpha x^2}\{\alpha  \Gamma_{12}{}^1(x^1) v(x^1)+\Gamma_{22}{}^1(x^1) v^\prime(x^1)\},
\text{ so }\Gamma_{12}{}^1=0.\\
&&K_{21}{}^1:\ 0=e^{\alpha x^2}\{\alpha\Gamma_{21}{}^1(x^1)v(x^1)
+\Gamma_{22}{}^1(x^1)v^\prime(x^1)\}, \text{ so }\Gamma_{21}{}^1=0.\\
&&K_{22}{}^2:\ 0=e^{\alpha x^2}\{\alpha  v(x^1) (\alpha +\Gamma_{22}{}^2(x^1))- \Gamma_{22}{}^1(x^1) v^\prime(x^1)\},
\text{ so }\Gamma_{22}{}^2=-\alpha\,.
\end{eqnarray*}
\smallbreak\noindent{\bf Case 1.2.} Suppose $\alpha=0$. The assumption that 
$X=v(x^1)\partial_{x^2}$ is not a constant multiple of
 $\partial_{x^2}$ implies $v$ is non-constant so
$v^\prime\ne0$. Equation~(\ref{E3.a}) follows from the equations
$$\begin{array}{lll}
K_{11}{}^1\!:\ 0=(\Gamma_{12}{}^1(x^1)+\Gamma_{21}{}^1(x^1)) v^\prime(x^1),\!&\!
K_{12}{}^1\!:\ 0=\Gamma_{22}{}^1(x^1) v^\prime(x^1),\\[0.05in]
K_{12}{}^2\!:\ 0=(\Gamma_{22}{}^2(x^1)-\Gamma_{12}{}^1(x^1)) v^\prime(x^1),\!&\!
K_{21}{}^2\!:\ 0=(\Gamma_{22}{}^2(x^1)-\Gamma_{21}{}^1(x^1))v^\prime(x^1)\,.
\end{array}$$

Assume Equation~(\ref{E3.a}) holds. Assertion~(2-a) follows from the affine
Killing equations
$K_{11}{}^2$: $0=(\Gamma_{12}{}^2+\Gamma_{21}{}^2)w^{(1,0)}+w^{(2,0)}$ and
$K_{22}{}^1$: $0=\alpha\ u^{(0,1)}+u^{(0,2)}$.
Assertion~(2-b) follows as $K_{11}{}^2$: $0=e^{\alpha x^2}((\Gamma_{12}{}^2+\Gamma_{21}{}^2)v^\prime+v^{\prime\prime})$ is
 the only non-trivial affine Killing equation for $e^{\alpha x^2}v(x^1)\partial_{x^2}$.
To prove Assertion~(2-c), assume that
 $u(x^1,x^2)\partial_{x^1}+\{\sum_nw_n(x^1)(x^2)^n\}\partial_{x^2}\in\mathfrak{K}_{\mathbb{C}}(\mathcal{M})$.
By Assertion~(2-a-ii), 
$\sum_n\{(\Gamma_{12}{}^2(x^1)+\Gamma_{21}{}^2(x^1))w_n^\prime(x^1)+w_n^{\prime\prime}(x^1)\}(x^2)^n=0$.
Thus each $w_n(x^1)$ satisfies
the ODE individually so by Assertion~(2-b), $w_n(x^1)\partial_{x^n}\in\mathfrak{K}_0(\mathcal{M})$.
One verifies directly that $x^2\partial_{x^2}$ is an affine Killing vector field in this setting.
\end{proof}

We use Lemma~\ref{L3.2} to study a Lie algebra structure which will arise subsequently.

\begin{lemma}\label{L3.3}
Suppose there is a 3-dimensional effective Lie subalgebra of $\mathfrak{K}(\mathcal{M})$ 
which is spanned by vector fields $X,Y,Z$ satisfying $[X,Y]=0$, $[X,Z]=aX+Y$, and $[Y,Z]=aY-X$ for $a\in\mathbb{R}$. Then there exists an effective Lie subalgebra of $\mathfrak{K}(\mathcal{M})$
isomorphic to $\mathfrak{K}_{\mathcal{A}}$.
\end{lemma}
\begin{proof}

 \smallbreak The Lemma is immediate if $\{X,Y\}$ is effective.
Consequently, we assume that $Y$ is a multiple of $X$ and $\{X,Z\}$ is effective.
Normalize the coordinate system as in
Lemma~\ref{L3.1} so that $X=\partial_{x^2}$ and thus $Y=v(x^1,x^2)\partial_{x^2}$. 
Since $[X,Y]=0$, $\partial_{x^2}v=0$ and thus
$v=v(x^1)$.  As $Y$ is not a constant multiple of $X$,  $v^\prime(x^1)\ne0$.
By Lemma~\ref{L3.2}~(1), the relations of Equation~(\ref{E3.a}) hold with $\alpha=0$. 
By Lemma~\ref{L3.2}~(2-c),
$x^2\partial_{x^2}\in\mathfrak{K}(\mathcal{M})$.
Expand $Z=u(x^1,x^2)\partial_{x^1}+w(x^1,x^2)\partial_{x^2}$.
We have
\begin{eqnarray*}[X,Z]&=&\partial_{x^2}u(x^1,x^2)\partial_{x^1}+\partial_{x^2}w(x^1,x^2)\partial_{x^2}\\
&=&aX+Y=(a+v(x^1))\partial_{x^2}\,.
\end{eqnarray*}
Thus $u=u(x^1)$ and $w=(a+v(x^1))x^2+v_0(x^1)$. As $\{X,Z\}$ is effective, $u\ne0$. By Lemma~\ref{L3.2}~(2-c),
$v_0(x^1)\partial_{x^2}\in\mathfrak{K}_0(\mathcal{M})$. Thus
$\tilde Z:=u(x^1)\partial_{x^1}+(a+v(x^1)x^2)\partial_{x^2}$ belongs to $\mathfrak{K}(\mathcal{M})$. 
Since $[\tilde Z,x^2\partial_{x^2}]=0$,
$\operatorname{Span}\{\tilde Z,x^2\partial_{x^2}\}$ is an effective Lie subalgebra of $\mathfrak{K}(\mathcal{M})$
which is isomorphic to $\mathfrak{K}_{\mathcal{A}}$.\end{proof}

\subsection{The proof of Lemma~\ref{L2.2}~(1)}
Use Lemma~\ref{L3.1} to normalize the coordinate system. Choose $X\in E(\alpha)$ for some $\alpha$
so $\{X,\partial_{x^2}\}$ is effective. Expand $X$ in the form given in Equation~(\ref{E2.b}). Since $\{X,\partial_{x^2}\}$ is effective, $u_i\ne0$ for some $i$. Choose $i_0$ maximal so $u_{i_0}\ne0$.
Apply $(\ad(\partial_{x^2})-\alpha)^{i_0}$ to $X$ to assume that $i_0=0$ so
\begin{equation}\label{E3.b}
X=e^{\alpha x^2}\left\{u(x^1)\partial_{x^1}+\sum_{j=0}^{j_0}v_j(x^1)(x^2)^j\partial_{x^2}\right\}\text{ for }u\ne0\,.
\end{equation}

We first examine $E(\alpha)$ for $\alpha\ne0$.

\begin{lemma}\label{L3.4}If $\alpha\ne0$, then there exists an
effective Lie subalgebra of $\mathfrak{K}{(\mathcal{M})}$ isomorphic to $\mathfrak{K}_{\mathcal{A}}$,
$\mathfrak{K}_{\mathcal{B}}$, or $\mathfrak{so}(3)$.
\end{lemma}

\begin{proof}Adopt the notation established above.
We wish to show $j_0=0$. Suppose to the contrary that $v_j\ne0$ for some $j>0$. Choose $j_0$ maximal so $v_{j_0}\ne0$ and hence
$$0\ne(\ad(\partial_{x^2})-\alpha)^{j_0}X=j_0!e^{\alpha x^2}v_{j_0}(x^1)\partial_{x^2}
\in\mathfrak{K}_{\alpha}(\mathcal{M})\,.$$
By Lemma~\ref{L3.2}~(1), Equation~(\ref{E3.a}) holds. 
We apply Lemma~\ref{L3.2}~(2-a-i) to see $2\alpha^2\ e^{\alpha x^2}u(x^1)=0$ so $u=0$
contrary to our assumption. Thus we conclude $j_0=0$
and $X=e^{\alpha x^2}\{u(x^1)\partial_{x^1}+v(x^1)\partial_{x^2}\}$.
\smallbreak\noindent{\bf Case 1.} Suppose $\alpha\in\mathbb{R}$. 
 $[\partial_{x^2},X]=\alpha X$ so $[\alpha^{-1}\partial_{x^2},X]=X$. Since
$\{X,\partial_{x^2}\}$ is effecttive, we have an effective Lie subalgebra isomorphic to $\mathfrak{K}_{\mathcal{B}}$.
We therefore assume $\alpha\in\mathbb{C}-\mathbb{R}$. By rescaling $x^2$, we may suppose 
$\alpha=a+\sqrt{-1}$ for $a\ge0$.
\smallbreak\noindent{\bf Case 2.} Assume
$a\ne0$. Choose $a$ maximal so there exists $X\in E(a+\sqrt{-1})$ so $\{X,\partial_{x^2}\}$ is effective. Expand
$X=e^{ax^2}e^{\sqrt{-1} x^2}\{u(x^1)\partial_{x^1}+v(x^1)\partial_{x^2}\}$.
We have $\bar X\in E(\bar\alpha)$. Let $Y_1:=\sqrt{-1}[X,\bar X]$.  By Lemma~\ref{L2.3},  $Y_1\in E(2a)$.
Since $\bar Y_1=Y_1$, $Y_1$ is real.
Decompose
$Y_1=e^{2a x^2}\{u_1(x^1)\partial_{x^1}+v_1(x^1)\partial_{x^2}\}$.
\smallbreak\noindent{\bf Case 2.1.} If $u_1\ne0$, then we may apply Case 1 to $Y_1$.
\smallbreak\noindent{\bf Case 2.2.} If $u_1=0$ and if $v_1\ne0$, then we apply Lemma~\ref{L3.2}~(1)
to see that Equation~(\ref{E3.a}) holds with $\Gamma_{22}{}^2=-2a$.
We apply Lemma~\ref{L3.2}~(2-a-i) to $X$ to see $(3a^2-1+4a\sqrt{-1})e^{\alpha x^2}u(x^1)=0$
so $u=0$ contrary to our assumption.
\smallbreak\noindent{\bf Case 2.3.} If $u_1=0$ and $v_1=0$, then $[X,\bar X]=0$ and Lemma~\ref{L3.3} pertains
with respect to the Lie algebra $\{\Re(X),\Im(X),\partial_{x^2}\}$, since
\begin{eqnarray*}
&&[\Im(X),\Re(X)]=0,\quad [\Im(X),-\partial_{x^2}]=a\Im(X)+\Re(X),\\
&&[\Re(X),-\partial_{x^2}]=a\Re(X)-\Im(X)\,.
\end{eqnarray*}
\medbreak\noindent{\bf Case 3.} Suppose $\alpha=\sqrt{-1}$. We have $X_i$ in $\mathfrak{K}(\mathcal{M})$ with $\{X_i,\partial_{x^2}\}$ effective where
\begin{eqnarray*}
&&X_1=u(x^1,x^2)\partial_{x^1}+v(x^1,x^2)\partial_{x^2},\quad X_2=\ad(\partial_{x^2})X_1,\\
&&u(x^1,x^2)=u_1(x^1)\cos(x^2)+u_2(x^1)\sin(x^2),\\
&&v(x^1,x^2)=v_1(x^1)\cos(x^2)+v_2(x^1)\sin(x^2)\,.
\end{eqnarray*}
Since $\{X_1,\partial_{x^2}\}$ is effective, $u\ne0$.
Let $X_3:=[X_1,X_2]\in E(0)$. Because there are no powers of $x^2$ in $X_1$ or $X_2$, we have that
 $X_3=u_3(x^1)\partial_{x^1}+v_3(x^1)\partial_{x^2}$.
\smallbreak\noindent{\bf Case 3.1.} If $u_3\ne0$, then $\{X_3,\partial_{x^2}\}$ is an effective Lie algebra isomorphic
to $\mathfrak{K}_{\mathcal{A}}$.
\smallbreak\noindent{\bf Case 3.2.} If $u_3=0$ but $v_3\ne0$, then
$X_3=v_3(x^1)\partial_{x^2}$. If $v_3^\prime\ne0$, we may apply Lemma~\ref{L3.2}~(1)
to obtain the relations of Equation~(\ref{E3.a}) with $\alpha=0$. 
 We may then apply Lemma~\ref{L3.2}~(2-a-i)
to see $u^{(0,2)}=0$. Since $u=-u^{(0,2)}$, $u=0$ contrary to our assumption. 
Thus $v_3^\prime=0$ and $[X_1,X_2]$ is a constant non-zero multiple of $\partial_{x^2}$. This
gives the Lie algebra $\mathfrak{so}(3)$.
\smallbreak\noindent{\bf Case 3.3.}  
If $X_3=0$, we have $[X_1,X_2]=0$ and we can apply Lemma~\ref{L3.3}.
\end{proof}

We now examine $E(0)$.
\begin{lemma}\label{L3.5}
Assume there exists $X\in E(0)$ such that $\{X,\partial_{x^2}\}$ is effective.
Then there exists an effective Lie subalgebra $\mathfrak{K}_0\subset\mathfrak{K}(\mathcal{M})$
isomorphic to $\mathfrak{K}_\mathcal{A}$ or $\mathfrak{K}_{\mathcal{B}}$.\end{lemma}

\begin{proof} Choose $X\in E(0)$ of the form given in Equation~(\ref{E3.b}) with $\alpha=0$.
If $j_0=0$, then $\{X,\partial_{x^2}\}$ is an effective algebra isomorphic to $\mathfrak{K}_{\mathcal{A}}$.
We may therefore assume that $j_0\ge1$. We suppose $j_0\ge2$ and argue for a contradiction.
Since $j_0-1\le 2j_0-3$, $u(x^1)\partial_{x^1}$ contributes lower order terms and
plays no role. Set:
\begin{eqnarray*}
&&Y_1:=\ad(\partial_{x^2})X=\{ c_1v_{j_0}(x^1)(x^2)^{j_0-1}+O((x^2)^{j_0-2})\}\partial_{x^2},\\
&&Y_2:=[X,Y_1]=\{c_2v_{j_0}^2(x^1)(x_2)^{2(j_0-1)}+O((x^2)^{2(j_0-1)-1})\}\partial_{x^2},\\
&&\dots\\
&&Y_n:=[X,Y_{n-1}]=\{c_nv_{j_0}^n(x^1)(x^2)^{n(j_0-1)}+O((x^2)^{n(j_0-1)-1})\}\partial_{x^2}\,.
\end{eqnarray*}
One verifies all the normalizing constants $c_n$ are non-zero so
 creates an infinite string of linearly independent elements of $\mathfrak{K}(\mathcal{M})$ which is not possible.
We therefore  suppose $j_0=1$ henceforth so
$X=u(x^1)\partial_{x^1}+(v_1(x^1)x^2+v_0(x^1))\partial_{x^2}$ for $v_1\ne0$.
 If $v_1^\prime=0$, then $ [\partial_{x^2},X]=v_1\partial_{x^2}$ is a constant multiple of $\partial_{x^2}$
  and we obtain
 a subalgebra isomorphic to $\mathfrak{K}_{\mathcal{B}}$.  
 We therefore suppose that
 $ v_1^\prime\ne0$ and apply Lemma~\ref{L3.2}~(1) to obtain the relations of Equation~(\ref{E3.a}) with $\alpha=0$.
 By Lemma~\ref{L3.2}~(2-c),   $x^2\partial_{x^2}\in\mathfrak{K}(\mathcal{M})$ and
 $v_0(x^1)\partial_{x^2}\in\mathfrak{K}_0(\mathcal{M})$.
Consequently we have $\tilde X:=u(x^1)\partial_{x^1}+v_1(x^1)x^2\partial_{x^2}\in\mathfrak{K}(\mathcal{M})$. 
 We have $[\tilde X,x^2\partial_{x^2}]=0$ and thus $\operatorname{Span}\{\tilde X,x^2\partial_{x^2}\}$ is an
 effective Lie subalgebra of $\mathfrak{K}(\mathcal{M})$ isomorphic to $\mathfrak{K}_{\mathcal{A}}$.
 This completes the proof of Lemma~\ref{L3.5} and thereby of Lemma~\ref{L2.2}~(1).
\end{proof}

\subsection{The proof of Lemma~\ref{L2.2}~(2)}
Let $\{X,Y\}$ be affine Killing vector fields which are effective and which satisfy $[X,Y]=0$.
The Frobenius Theorem lets us choose local coordinates so $X=\partial_{x^1}$ and $Y=\partial_{x^2}$;
we then have $\Gamma_{ij}{}^k\in\mathbb{R}$. \qed
\subsection{The proof of Lemma~\ref{L2.2}~(3)}
The following is a useful observation.
\begin{ansatz}\label{A3.6}
\rm Let $X=u(x^1)\partial_{x^1}+(x^2+v(x^1))\partial_{x^2}$ where $u\ne0$. 
Set $\tilde x^1=x^1$ and $\tilde x^2=x^2+\varepsilon(x^1)$. Then
$\partial_{\tilde x^1}=\partial_{x^1}-\varepsilon^\prime(x^1)\partial_{x^2}$ and $\partial_{\tilde x^2}=\partial_{x^2}$.
We then have
$X=u(\tilde x^1)\partial_{\tilde x^1}+\{\tilde x^2-\varepsilon(x^1)+v(x^1)+u(x^1)\varepsilon^\prime(x^1)\}\partial_{\tilde x^2}$.
We may then solve the ODE $-\varepsilon(x^1)+v(x^1)+u(x^1)\varepsilon^\prime(x^1)=0$ to express
$
X=u(\tilde x^1)\partial_{\tilde x^1}+\tilde x^2\partial_{\tilde x^2}$ where 
$w(\tilde x^1)\partial_{\tilde x^2}=w(x^1)\partial_{x^2}$ for any $w$.
\end{ansatz}

Let $\{X,Y\}$ be affine Killing vector fields which are effective with $[X,Y]=Y$. 
Choose local coordinates so $Y=\partial_{x^2}$. 
Expand $X=u(x^1,x^2)\partial_{x^1}+v(x^1,x^2)\partial_{x^2}$. Since $[X,Y]=Y$, $\partial_{x^2}u=0$ and $\partial_{x^2}v=-1$ so $X=u(x^1)\partial_{x^1}+(-x^2+v_0(x^1))\partial_{x^2}$. 
Use Ansatz~\ref{A3.6} to change
coordinates so
$X=u(x^1)\partial_{x^1}-x^2\partial_{x^2}$ without changing $Y=\partial_{x^2}$. 
Replace $x^1$ by $\hat x^1$ to ensure 
$u(x^1)\partial_{x^1}=-\hat x^1\partial_{\hat x^1}$ so $X=-\hat x^1\partial_{\hat x^1}-\hat x^2\partial_{\hat x^2}$. 
The Killing equations for $X$ yield $K_{ij}{}^k$: $0=\Gamma_{ij}{}^k(\hat x^1)+\hat x^1  \Gamma_{ij}^\prime{}^k(\hat x^1)$ for
$i,j,k=1,2$. This shows that the Christoffel symbols have the desired form.
\qed

\subsection{The proof of Lemma~\ref{L2.2}~(4)} Throughout this section, we will not
use the normalizations of Lemma~\ref{L3.1}. We shall, however, assume always
that $\partial_{x^2}\in\mathfrak{K}(\mathcal{M})$ so $\Gamma_{ij}{}^k=\Gamma_{ij}{}^k(x^1)$.
We begin by showing:

\begin{lemma}\label{L3.7} Suppose $\mathcal{M}$ has an effective Lie subalgebra isomorphic to $\mathfrak{so}(3)$.
Then the connection is torsion free, the Ricci tensor $\rho$ is positive definite and symmetric, and $\nabla\rho=0$.
\end{lemma}

\begin{proof} Let $\operatorname{Span}\{X,Y,Z\}$ be an effective Lie subalgebra of
$\mathfrak{K}(\mathcal{M})$ satisfying the relations
of Equation~(\ref{E2.a}) defining $\mathfrak{so}(3)$.
Choose coordinates so $Z=\partial_{x^2}$. Decompose $X=u(x^1,x^2)\partial_{x^1}+v(x^1,x^2)\partial_{x^2}$. We then
have $u^{(0,2)}=- u$. We may then express $u(x^1,x^2)=r(x^1)\cos(x^2+\theta(x^1))$ where by hypothesis
$r(x^1)\ne0$. Use Ansatz~\ref{A3.6} to replace $x^2$ by $x^2+\theta(x^1)$ and rewrite
$X=r_1(y^1)\cos(y^2)\partial_{y^1}+v_1(x^1,x^2)\partial_{y^2}$ without changing $\partial_{x^2}$. Choose coordinates
$(z^1,z^2)=(f(y^1),y^2)$ so that $\partial_{z^1}=r_1(y^1)\partial_{y^1}$ and $\partial_{z^2}=\partial_{y^2}$. Since 
$v_1^{(0,2)}=-v_1$, the bracket relation $[Z,X]=Y$ gives
\begin{eqnarray*}
&&X=\cos(z^2)\partial_{z^1}+\{v_c(z^1)\cos(z^2)+v_s(z^1)\sin(z^2)\}\partial_{z^2},\\ 
&&Y=-\sin(z^2)\partial_{z^1}+\{-v_c(z^1)\sin(z^2)+v_s(z^1)\cos(z^2)\}\partial_{z^2}\,.
\end{eqnarray*}
The bracket relation $[X,Y]=Z$ now yields $-v_c(x^1)^2-v_s(x^1)^2+v_s^\prime(x^1)=1$ and
$v_c(x^1)=0$.
We solve this to obtain $v_c(x^1)=0$ and $v_s(x^1)=\tan(x^1+c)$; we can further normalize the coordinates so $c=0$.
Thus, after a suitable change of notation, we have $Z=\partial_{x^2}$,
\begin{eqnarray*}
&&X=\cos(x^2)\partial_{x^1}+\tan(x^1)\sin(x^2)\partial_{x^2},\\
&&Y=-\sin(x^2)\partial_{x^1}+\tan(x^1)\cos(x^2)\partial_{x^2}.
\end{eqnarray*}
We have $\Gamma_{ij}{}^k=\Gamma_{ij}{}^k(x^1)$. We evaluate the affine Killing equations corresponding to $X$
at $x^2=0$ to obtain 
\medbreak\quad
$K_{11}{}^1:\ 0=\Gamma_{11}^\prime{}^1(x^1)$,\hfill
$K_{11}{}^2:\ 0=-\Gamma_{11}{}^2(x^1)\tan(x^1)+\Gamma_{11}^\prime{}^2(x^1)$,
\smallbreak\quad
$K_{12}{}^1:\ 0=\Gamma_{12}{}^1(x^1)\tan(x^1)+\Gamma_{12}^\prime{}^1(x^1)$,\hfill
$K_{12}{}^2:\ 0=\sec^2(x^1)+\Gamma_{12}^\prime{}^2(x^1)$,
\smallbreak\quad
$K_{21}{}^1:\ 0=\Gamma_{21}{}^1(x^1)\tan(x^1)+\Gamma_{21}^\prime{}^1(x^1)$,\hfill
$K_{212}:\ 0=\sec^2(x^1)+\Gamma_{21}^\prime{}^2(x^1)$,
\smallbreak\quad
$K_{22}{}^1:\ 0=-1+2\Gamma_{22}{}^1(x^1)\tan(x^1)+\Gamma_{22}^\prime{}^1(x^1)$,
\smallbreak\quad
$K_{22}{}^2:\ 0=\Gamma_{22}{}^2(x^1)\tan(x^1)+\Gamma_{22}^\prime{}^2(x^1)$.
\medbreak\noindent We solve these ODEs to obtain real constants $a_{ij}{}^k$ so that
$$\begin{array}{ll}
\Gamma_{11}{}^1(x^1)=a_{11}{}^1,&\Gamma_{11}{}^2(x^1)=a_{11}{}^2\sec(x^1),\\[0.05in]
\Gamma_{12}{}^1(x^1)=a_{12}{}^1\cos(x^1),&\Gamma_{12}{}^2(x^1)=a_{12}{}^2-\tan(x^1),\\[0.05in]
\Gamma_{21}{}^1(x^1)=a_{21}{}^1\cos(x^1),&\Gamma_{21}{}^2(x^1)=a_{21}{}^2-\tan(x^1),\\[0.05in]
\Gamma_{22}{}^1(x^1)=a_{22}{}^1\cos(x^1)^2+\cos(x^1)\sin(x^1),&
\Gamma_{22}{}^2(x^1)=a_{22}{}^2\cos(x^1).
\end{array}$$
Let $i\ne j$ and $1\le i,j\le2$. We evaluate the affine Killing equations for $X$ at $(x^1,x^2)=(0,\frac\pi2)$ to obtain
$$\begin{array}{ll}
K_{ii}{}^i:\ 0=a_{ii}{}^j + a_{ij}{}^i + a_{ji}{}^i,&
K_{ii}{}^j:\ 0=-a_{ii}{}^i + a_{ij}{}^j + a_{ji}{}^j\,,\\
K_{ij}{}^i:\ 0=-a_{ii}{}^i + a_{ij}{}^j + a_{jj}{}^i,&
K_{ij}{}^j:\ 0=-a_{ii}{}^j -a_{ij}{}^i + a_{jj}{}^j.\\
\end{array}$$ 
These equations imply that all the $a_{ij}{}^k$ vanish and thus the non-zero Christoffel
symbols are
$\Gamma_{12}{}^2(x^1)=\Gamma_{21}{}^2(x^1)=-\tan(x^1)$ and
$\Gamma_{22}{}^1(x^1)=\cos(x^1)\sin(x^1)$. We complete the proof of the Lemma by computing that
$\rho=(dx^1)^2+\cos(x^1)^2(dx^2)^2$ and $\nabla\rho=0$.\qed

\medbreak We apply Lemma~\ref{L3.7}. We have shown $\nabla$ is torsion free. 
Let $ds^2=\rho$. We have $\nabla\rho=0$. This shows $\nabla$ is the 
Levi-Civita connection of $ds^2$; this is a positive definite metric with Einstein constant 1.
Thus this geometry is modeled on the round sphere.
This completes the proof of Lemma~\ref{L2.2} and thereby of Theorem~\ref{T1.1}. 
\end{proof}

\end{document}